\documentclass[reqno,11pt]{amsart}
\usepackage{amsmath,amsthm,amssymb,amsfonts,xspace}
\usepackage[english]{babel}
\usepackage{mathrsfs}
\usepackage[usenames,dvipsnames]{xcolor}
\usepackage{xcolor}
\usepackage[utf8]{inputenc}
\usepackage{amssymb,xspace}
\usepackage[all]{xy}
\usepackage[bookmarksnumbered,colorlinks]{hyperref}
\usepackage{graphicx}
\usepackage{enumerate}
\usepackage{enumitem}
\usepackage[normalem]{ulem}% overlines
\usepackage[colorinlistoftodos,italian]{todonotes}

\usepackage[ddmmyyyy,hhmmss]{datetime}

\newtheorem{thm}{Theorem}[section]
\newtheorem{prop}[thm]{Proposition}
\newtheorem{lem}[thm]{Lemma}
\newtheorem{cor}[thm]{Corollary}
\theoremstyle{definition}
\newtheorem{dfn}[thm]{Definition}
\newtheorem{rmk}[thm]{Remark}
\newtheorem{ass}[thm]{Assumption}

\def\bal#1\eal{\begin{align}#1\end{align}}              
\def\baln#1\ealn{\begin{align*}#1\end{align*}}          
\def\bml#1\eml{\begin{multline}#1\end{multline}}        
\def\bmln#1\emln{\begin{multline*}#1\end{multline*}}  
\def\bga#1\ega{\begin{gather}#1\end{gather}}
\def\bgan#1\egan{\begin{gather*}#1\end{gather*}}

\makeatletter
\newcommand*{\bigcdot}{}% Check if undefined
\DeclareRobustCommand*{\bigcdot}{%
	\mathbin{\mathpalette\bigcdot@{}}%
}
\newcommand*{\bigcdot@scalefactor}{.7}
\newcommand*{\bigcdot@widthfactor}{1.15}
\newcommand*{\bigcdot@}[2]{%
	\sbox0{$#1\vcenter{}$}% math axis
	\sbox2{$#1\cdot\m@th$}%
	\hbox to \bigcdot@widthfactor\wd2{%
		\hfil
		\raise\ht0\hbox{%
			\scalebox{\bigcdot@scalefactor}{%
				\lower\ht0\hbox{$#1\bullet\m@th$}%
			}%
		}%
		\hfil
	}%
}
\makeatother

\newcommand{\X}{\mathcal X_\varphi\xspace}
\newcommand{\W}{\ensuremath{\stackrel{\bigcdot}{W}\!\!\mbox{}^{1,2}}}
\newcommand{\de}{{\rm d}}
\newcommand{\R}{\ensuremath{\mathbb R}\xspace}
\newcommand{\N}{\ensuremath{\mathbb N}\xspace}

\newcommand{\LL}{\ensuremath{\mathbb L}\xspace}

\newcommand{\beq}{\begin{equation}}
\newcommand{\eeq}{\end{equation}}
\newcommand{\bere}{\begin{rmk}}
\newcommand{\ere}{\end{rmk}}

\def\br#1\er{\textcolor{red}{#1}} 
\def\bl#1\el{\textcolor{blue}{#1}} 

\title[Spacelike graphs with prescribed mean curvature]{Spacelike graphs  with prescribed mean curvature  on exterior domains in the Minkowski spacetime} 
\author[R. Bartolo]{Rossella Bartolo}

\address{Dipartimento di Meccanica, Matematica e Management, \hfill\break\indent
	Politecnico di Bari, Via Orabona 4, 70125, Bari, Italy \vspace{0.5cm}}

\email{rossella.bartolo@poliba.it}
\author[E. Caponio]{Erasmo Caponio}
\email{erasmo.caponio@poliba.it}
\author[A. Pomponio]{Alessio Pomponio}
\email{alessio.pomponio@poliba.it}

\thanks{This work is partially supported by  PRIN 2017JPCAPN {\em Qualitative and quantitative aspects of nonlinear
PDEs}}


\begin{document}

\begin{abstract}
 We consider a Dirichlet problem for the mean curvature operator in  the Minkowski spacetime,   obtaining  a necessary and sufficient condition for the   existence of a spacelike solution,  with prescribed mean curvature, 
which is the   graph of a function   defined  
on a   domain  equal to  the complement in $\R^n$ of the union of a finite number of bounded Lipschitz domains. The mean curvature $H=H(x,t)$ is assumed to have   absolute value  controlled from above by a locally bounded, $L^p$-function, $p\in [1,2n/(n+2)]$,  $n\geq 3$.  
\end{abstract}
\maketitle

\section{Introduction and statement of the results}
Due to their importance in general relativity, spacelike hypersurfaces with constant mean curvature or, more generally, with prescribed mean curvature $H$, have been extensively studied since Lichnerowicz's paper \cite{Li44} (cf.  \cite{MaTi80}). Maximal spacelike hypersurfaces ($H=0$) have also attracted the interest of researchers because of  their similarity with minimal hypersurfaces in the Riemannian setting. 
In fact, they are  critical points of the area functional and,  under some curvature assumptions on the ambient spacetime, locally maximize it among all nearby spacelike hypersurfaces having the same boundary (cf. \cite{BrFl76}).  Two  very relevant examples  in this perspective  are   the so-called Calabi-Bernstein problem in the Minkowski spacetime \cite{Ca70, ChYa76} and 
its counterpart with constant mean curvature hypersurfaces \cite{Treibe82}.  
 Since then a lot of related papers have been  appearing; just to recall a few, we refer here to  the  new proofs of the Calabi-Bernstein theorem for surfaces given  in \cite{Ro96, AlPa01},  as well as  to   related Calabi-Bernstein type results in different ambient spaces \cite{EckHui91, AlRoSa95, AlbAl09, CaRoRu11, CaCaLi11} and  to  some existence results in Minkowski spacetime for entire or radial spacelike graphs  under different growth conditions on the mean curvature \cite{BDD,A,A2,P,MP}.

The Dirichlet problem on a bounded open set  for spacelike hypersurfaces, described  as the graph of a function,  with prescribed mean curvature  was studied in the Minkowski spacetime \cite{BarSim83}, in spacetimes conformal to an orthogonal  splitting \cite{Gerhar83} and  for some cosmological spacetimes \cite{Ba84}.  A fairly general case  was considered in \cite{Bartni88}. 
Recently, non-smooth critical point theory \cite{Sz86} has been used in \cite{BeJeMa14} to obtain existence and multiplicity of  spacelike solutions in the Minkowski spacetime for the Dirichlet problem  with homogeneous boundary data on a $C^2$ domain,  when the mean curvature is a  function depending on a parameter.
We would like to emphasize that the result in \cite{BarSim83}  does not require assumptions on the regularity of the boundary. Namely, boundary values are considered according to the definition given in \cite[p. 133]{BarSim83} which allows  one   dealing with quite general open bounded subsets of $\R^n$ (for instance bounded domains  with just continuous boundaries are admitted). 

Among other results, in \cite{BarSim83} it is also proved  that if a variational solution, with mean curvature not depending on the time coordinate $t$, contains a segment of light ray, then it contains the ray extended to the boundary or to infinity. 
This property will be fundamental in the proof of our main result,   Theorem~\ref{maint},  dealing with solutions  vanishing  at infinity of the Dirichlet problem for spacelike  hypersurfaces on an unbounded open subset of $\R^n$.  Theorem~\ref{maint}  can  be considered as an extension of \cite[Theorem 4.1]{BarSim83} to  exterior  domains. A solution is indeed obtained  by a standard minimization argument  (Proposition~\ref{pr:min})  applied to  functional $\mathcal I$ in \eqref{I} which, differently from the area functional, is well-defined  on functions with  square integrable gradient on the exterior domain (see Section~\ref{s3}).
The boundary conditions are as usual encoded in the functional setting adopted (see Section~\ref{s2}). As we need    an extendibility property  (Lemma~\ref{Lext}),   we  reinforce  a bit the regularity of the domain  w.r.t.  \cite{BarSim83}, by considering a Lipschitz boundary. Functional  $\mathcal I$  is then defined on a convex, closed subset (Proposition~\ref{weakly})  of an affine subspace of the homogeneous Sobolev space of locally integrable functions with square integrable partial derivatives which share the same trace on the boundary of the exterior domain. The tangent space of this affine manifold  is the Hilbert space obtained as the completion,   w.r.t.  the $L^2$-norm  of the gradient, of the space of test functions on the exterior domain. This fact allows us to recover  some   embedding  properties  (Lemma~\ref{prop})  in a similar variational setting exploited in  \cite{BoPo16}   to  find   solutions of a  Born-Infeld equation in $\R^n$  (see the end of this introduction for more details). 

Let $(\LL^{n+1}, \langle\cdot,\cdot\rangle)$ be the $(n+1)$-dimensional Minkowski spacetime with the following sign convention:  $\langle(\tau,v),(\tau,v)\rangle =-\tau^2+|v|^2$ for  $\tau\in \R$,  $v\in \R^n$, where $|\cdot |$ is the Euclidean norm in $\R^n$. A smooth immersion $\phi\colon\Sigma\rightarrow \LL^{n+1}$ of an $n$-dimensional connected manifold $\Sigma$ is a {\em spacelike hypersurface} if the metric induced by $\phi$ is a Riemannian metric on $\Sigma$.
Let $A$ be the shape operator of $\Sigma$ and  $H:= -\frac{1}{n}{\rm tr}A$ its  {\em mean curvature}.
Assume that $\Sigma$ is an open subset of $\R^n$ and $\phi(x)=\big(x,u(x)\big)$, $x\in\Sigma$; when $u$ is at least  of class  $C^2$, the mean curvature  of this  hypersurface is then equal to 
\[
{\rm div}\left(\frac{\nabla u}{\sqrt{1-|\nabla u|^2}}\right)\ =\ nH(x,u),
\]
where ${\rm div}(\cdot)$ is the divergence operator in $\R^n$.

Hereafter by a {\em domain} in $\R^n$ we mean an open and connected subset. Given a domain $\Omega$ and a $C^1$  function $u$ on it, its graph $t=u(x_1,\ldots,x_n)=u(x)$ defines a $C^1$ spacelike hypersurface in $\LL^{n+1}$ if and only if the Euclidean gradient of $u$ satisfies $|\nabla u(x)|<1$, for all $x\in \Omega$. In this case, we will say that  $u$ is a {\em  spacelike function} and its graph is then called a spacelike hypersurface. 

Instead, a  locally Lipschitz function $u$ is said {\em weakly spacelike} if $|\nabla u|\leq 1$, for {a.e.}  $x\in\Omega$.

We also need the following definition.

\begin{dfn}
Let $V\subset \R^n$; a  function $\psi\colon V\to \R$ is said {\em  spacelike displacing} if its graph is an acausal set,   namely no couple of its points  can be joined by a timelike or lightlike segment.   This is equivalent to  $|\psi(x)-\psi(y)|<|x-y|$, for all $x,y\in V$, $x\neq y$. Moreover, 
$\psi:V\to\R$ is said {\em spacelike displacing in $W\subset \R^n$} if $|\psi(x)-\psi(y)|<|x-y|$,  for all $x,y\in V$, $x\neq y$,  with the inner points of the line segment $\overline{xy}$ contained in $W$.  
\end{dfn}
\bere\label{nomi}
Notice that  in \cite{BarSim83}  $C^1$ spacelike  functions are named  {\em strictly spacelike}, while spacelike displacing functions in an open subset $V$ of $\R^n$ (and defined on the same $V$) are called {\em spacelike}  (although the graph of  a spacelike displacing function can have degenerate tangent spaces). Furthermore, note  that  if $V$ is open, the graph $G_\psi$ of a continuous  spacelike displacing function  in $V$, $\psi:V\to\R$, is spacelike in the usual sense for  $C^0$ hypersurfaces,   i.e., for each $p\in G_\psi$ there exists a neighbourhood $U$ in $\LL^{n+1}$ such that $G_\psi\cap U$ is acausal and edgeless in $U$ (cf., e.g.,  \cite[p. 213]{EscGal92}, \cite[Definition 14.28]{BeErEa96}). Nevertheless, if $V$ is not convex, the graph of a spacelike displacing $\psi:V\to\R$ in $V$ can be not acausal.
\ere

We recall the definition of a Lipschitz domain.

\begin{dfn}\label{lipdom}
An open subset  $U\subset\R^n$ is  said {\em Lipschitz} if for each $p\in \partial U$ there exist an open ball $B(p,r)\subset \R^n$ and a Lipschitz function $f\colon B(p,r) \to \R$ such that $B(p,r)\cap U =f^{-1}((0,+\infty))$.   
\end{dfn}
\bere\label{sepr}
If $U$ is a Lipschitz open subset, then $\partial U=\partial\overline U$ (see, e.g., \cite[Remark 9.59]{Leoni17}). For further use, notice that if  $U\subset \R^n$ is Lipschitz, then $\partial U=\partial (\R^n\setminus \overline U)$. 
\ere

We  deal  with Lipschitz exterior domains of $\R^n$. More precisely, we require the following assumption.

\begin{ass}\label{omega}
We  consider  an exterior domain in $\R^n$, $n\geq 3$, defined by means of a finite collection of  bounded Lipschitz domains $\Omega_i$, such that $\overline{\Omega}_i\cap \overline{\Omega}_j= \emptyset$ for all $i,j\in \{1,\ldots, m\}$, $i\neq j$, $m\geq 1$.
Let
\[
\Omega :=\bigcup_{i=1}^{m}\Omega_i \quad \text{and} \quad      \Omega_c:=\R^n\setminus\overline{\Omega}. \]
Notice that by Remark \ref{sepr} $\partial\Omega_c=\partial\Omega$.
\end{ass}

Next let us set our problem.  We consider  the Dirichlet problem
\begin{equation}\label{star}
\begin{cases}
\displaystyle {\rm div}\left(\frac{\nabla u}{\sqrt{1-|\nabla u|^2}}\right)\ =\ nH(x, u)
& \mbox{ in $\Omega_c$}\\
 \displaystyle u=\varphi & \mbox{ on $\partial\Omega$}  \\
  \displaystyle \lim_{|x|\rightarrow +\infty}  u(x)=0
 \end{cases}
\end{equation}
where $\varphi: \partial\Omega\rightarrow\R$ and
$H:\Omega_c\times\R\rightarrow\R$ is   a   Carath\'eodory function\footnote{$H(\cdot,t)$ is measurable in $\Omega_c$
 for all $t\in\R$ and $H(x,\cdot)$ is continuous in $\R$ for a.e. $x\in \Omega_c$.}  satisfying:
\begin{enumerate}[label=(H), ref=H]
\item \label{h}
there exists $h\in L^s(\Omega_c)
\cap L^\infty_{\rm loc}({\Omega}_c)$, 
$s\in \left[1,\frac{2n}{n+2}\right]$,
such that
$$
n|H(x,t)|\leq h(x) \quad\text{for a.e. $x\in \Omega_c$  and all $t\in\R$}.
$$
\end{enumerate}

In order to introduce our functional framework (see Section \ref{s2} for more details and remarks),  let us recall that   $\stackrel{\bigcdot }{W}\!\!\mbox{}^{1,2}(\Omega_c)$  is   the homogeneous Sobolev space  of  distributions on $\Omega_c$ with partial derivatives in $L^2(\Omega_c)$. Some useful properties of such a space can be found, for example, in \cite[\S 11]{Leoni17}. \\
Let us finally introduce the space $\mathcal X$ of admissible functions in the variational setting  for   problem \eqref{star}:
\begin{equation}\label{storto}
\mathcal X:=\ \W(\Omega_c)\cap L^{2^*}(\Omega_c)\cap \{u\in C^{0,1}_{\mathrm{loc}}(\Omega_c): \|\nabla u\|_\infty\leq 1\};
\end{equation}
here  $C^{0,1}_{\mathrm{loc}}(\Omega_c)$  denotes the space  of  locally Lipschitz functions on $\Omega_c$ and as usual $2^*=2n/(n-2)$ is the Sobolev critical exponent.

Now we give the definition of {\em weak solution} of \eqref{star} (see also Remark~\ref{rem:ws}). 

\begin{dfn}\label{def:ws}
A function $u: \Omega_c \to \R$ is called a {\em weak solution} of \eqref{star}  if it belongs to $ \mathcal X$, $\varphi$ is the trace of $u$ and  
\begin{equation}\label{ws}
\int_{\Omega_c}\frac{\nabla u \cdot \nabla v}{\sqrt{1-|\nabla u|^2}}\;{\rm d}x
+n\int_{\Omega_c}H(x,u)v\;{\rm d}x=0,
\quad \hbox{ for all }v\in C^\infty_c(\Omega_c).
\end{equation}
\end{dfn}

We state our main result.
\begin{thm}\label{maint}
Let $\Omega_c$ satisfy Assumption~\ref{omega} and  
$H:\Omega_c\times\R\rightarrow\R$ be such that \eqref{h} holds. 
Then, there exists a spacelike\footnote{ In the terminology of \cite{BarSim83}  such a solution  is strictly spacelike, cf. Remark~\ref{nomi}.}
weak solution    of \eqref{star}  if and only if 
$\varphi:\partial\Omega\rightarrow\R$  is  the trace 
of a function $w\in\mathcal{X}$ and moreover $\varphi$ is spacelike displacing in $\Omega_c$ when $\Omega$ is not convex. 
\end{thm}
\begin{rmk}
Let us emphasize some points about Theorem~\ref{maint}.
\begin{itemize}
	\item[(1)] The boundary condition, $u=\varphi$ on $\partial \Omega$, in \eqref{star} is meant in a trace  sense, but as we will show in the next section a weak solution  $u$  can   indeed be continuously Lipschitz extended to $\partial \Omega$ (Lemma \ref{Lext}). Therefore, $u_{|\partial\Omega_c}=\varphi$, hence $\varphi$ is a posteriori Lipschitz continuous on $\partial \Omega$.
	\item[(2)] The limit at infinity in \eqref{star} is intended in the classical sense.
	\item[(3)] 
 As shown in \cite[p. 148]{BarSim83}  (see also \cite[p. 5]{Bon-Iac2}), as $H$ is locally bounded, by elliptic regularity theory, a spacelike weak solution  $u$  belongs to $W^{2,2}_{\mathrm{loc}}(\Omega_c)$ and $\nabla u$  is locally H\"older. Moreover,  if $H\in C^{k,\alpha}(\Omega_c\times\R)$, $k\in\N\cup\{\infty\}$, then $u\in C^{k+2,\alpha}(\Omega_c)$. 
	\item[(4)] If $H$ does not depend on $t$, our statement holds just assuming that $H\in L^s(\Omega_c) \cap L^{\infty}_{\rm loc}(\Omega_c)$. In particular, it holds for $H=0$ and, in such a case, gives the existence of a maximal hypersurface on the exterior domain $\Omega_c$.
		Previous existence (and uniqueness, with respect to a given asymptotic profile) results  for the Dirichlet problem of  maximal graphs on an exterior domain of the Minkowski spacetime have been recently obtained in \cite{HonYua19}. Existence and multiplicity results for radial solutions outside a ball, with  homogeneous boundary condition, have been obtained in \cite{YaLeSi20} for a separable-variables $H$ which is also radial in the $x$ variable.
	We are not aware of  other  results for the Dirichlet problem in an exterior domain when $H\neq 0$. 
	\end{itemize}
\end{rmk}
 Since when $\varphi$ is $(1-\epsilon)$-Lipschitz continuous on $\partial\Omega$, for a $\epsilon>0$, we  can ensure that there exists $u\in\mathcal{X}$ such that  $u_{|\partial\Omega}=\varphi$,   we have the following: 
\begin{cor}\label{mainc}
Under the assumptions on $\Omega_c$ and $H$ in  Theorem~\ref{maint}, let  $\varphi:\partial\Omega\to \R$ be a $(1-\epsilon)$--Lipschitz continuous. Then, there exists a  spacelike weak   solution of \eqref{star}.	
\end{cor}

A further relevant physical motivation to the problem under study is given by the  differential operator
\begin{equation*}
\mathcal Q(u)={\rm div}\left(\dfrac{\nabla u}{\sqrt{1-|\nabla u|^2}}\right)
\end{equation*}
which naturally appears in the Born-Infeld theory.
Almost a century ago, Born and Infeld introduced a new electromagnetic theory in a series of papers \cite{Bnat,BInat,B,BI}
as an alternative to
the classical Maxwell theory. Such a theory was proposed  as a nonlinear model of electrodynamics  having the notable feature of being a fine answer to the well-known infinity energy problem   (the electromagnetic field generated by a point charge has finite energy  in Born-Infeld theory).

In last years many authors have been focusing their attention on problems related to  $\mathcal Q$ in the whole $\R^n$, $n\ge 1$. In particular, some results for
\[
-{\rm div}\left(\dfrac{\nabla u}{\sqrt{1-|\nabla u|^2}}\right)=\rho, \qquad\hbox{ in }\R^n
\]
can be found in \cite{K,BoPo16,K-corr,Bon-Iac2, BCF, BDPR}, under different assumptions on $\rho$. Here $\rho$ can be considered as an assigned charges source. We also refer to \cite{APS}, where the Born-Infeld equation is coupled with the nonlinear Schr\"odinger one.
 Therefore, Theorem \ref{maint} can also be seen  as an existence result for the Born-Infeld problem on an  exterior  domain with assigned boundary conditions.

\section{Functional setting}\label{s2}
Let us denote by  $\de x$  the Lebesgue measure on $\R^n$ and, unless differently specified, by $\|\cdot\|_r$  the $L^r$-norm in the Lebesgue space $L^r(\Omega_c)$, $1\leq r\leq +\infty$, where $\Omega_c$ is defined in Assumption~\ref{omega}. By $\|\nabla\cdot\|_r$ we mean $\||\nabla\cdot|\|_r$.

Let $V\subset\R^n$ be open and unbounded and let us consider the homogeneous Sobolev  space $D^{1,2}(V)$  defined as  the completion of  $C_c^\infty(V)$, the space of smooth functions with compact support in $V$,   with respect to the $L^2$-norm of the gradient. Moreover, let us denote by $\mathscr{D}'(V)$ the space of distribution on $V$.

The following properties  hold for $D^{1,2}(\Omega_c)$:
\begin{prop}\label{tr0} 
\mbox{}
\begin{enumerate}[label=(\arabic{*}), ref=\arabic{*}]
		\item \label{d1} $D^{1,2}(\Omega_c)\hookrightarrow L^{2^*}(\Omega_c)$; moreover 
		$D^{1,2}(\Omega_c)\hookrightarrow \mathscr{D}'(\Omega_c)$ and the partial distributional derivatives of its elements are represented by functions in $L^2(\Omega_c)$;
		\item \label{d2}all $u\in D^{1,2}(\Omega_c)$ vanish at infinity,  i.e., ${\rm meas}\{x\in \Omega_c:|u(x)|>t\}<+\infty$,  for all $t>0$;
	\item \label{d3}for all $u\in D^{1,2}(\Omega_c)$, the trace ${\rm Tr}(u)$  on $\partial\Omega_c$ is well-defined and equal to $0$.
\end{enumerate} 
\end{prop}
\begin{proof}
\eqref{d1} Since $n\geq 3$, by the Sobolev embedding theorem,  $D^{1,2}(\Omega_c)$ is continuously embedded in $L^{2^*}(\Omega_c)$ (see \cite[\S 15.1]{Mazya11}). As a consequence, 
every $u\in D^{1,2}(\Omega_c)$ can be identified with a distribution on $\Omega_c$ and, if $(u_k)_k\subset D^{1,2}(\Omega_c)$ is such that  $\|\nabla u_k-\nabla u\|_{2}\to 0$, then it converges to $u$ in distributional sense. Let now   $(\varphi_k)_k\subset C_c^\infty(\Omega_c)$ be a  Cauchy sequence representing $u\in D^{1,2}(\Omega_c)$; then $(\varphi_k)_k$  is Cauchy in $L^{2^*}(\Omega_c)$ and  for all 
$\varphi\in C_c^\infty(\Omega_c)$, $i\in\{1,\dots,n\}$,  we have 
\bmln   (\partial_i u,\varphi) := -\int_{\Omega_c} u\partial_i\varphi \,\de x=-\lim_k\int_{\Omega_c}  \varphi_k\partial_i\varphi \,\de x   \\   =\lim_k\int_{\Omega_c} \partial_i \varphi_k\varphi \,\de x=\int_{\Omega_c} \psi_i\varphi \,\de x,
\emln where $\psi_i\in L^2(\Omega_c)$ is the element representing the Cauchy sequence $\partial_i \varphi_k$ in $L^2(\Omega_c)$.  

\eqref{d2} Notice that $u\not\in L^{2^*}(\Omega_c)$ if, by contradiction, there exists $\bar t>0$ such that 	${\rm meas}\{x\in \Omega_c:|u(x)|>\bar t\}=+\infty$. 

\eqref{d3} Take $R>0$ such that $\overline{\Omega} \subset B(0,R)$ and consider $v$, the restriction to $\Omega_c$ of a smooth function assuming values in the interval $[0,1]$ and equal to $1$ on $B(0,R)$, with compact support on $B(0,2R)$ and such that $|\nabla v|<1$.
Let  $(u_k)_k\subset C_c^\infty(\Omega_c)$ be a sequence  converging to $u$ in $D^{1,2}(\Omega_c)$. Then $v u_k\to vu$ both in $D^{1,2}(\Omega_c)$ and $L^{2^*}(\Omega_c)$, thus   $vu_k\to vu$ in the $H^1$-norm on $B(0,2R)\setminus \overline{\Omega}$, which implies that 
$vu\in H^1_0(B(0,2R)\setminus \overline{\Omega})$. As $\partial\Omega_c$ is Lipschitz, $\mathrm{Tr}(vu)=0$ (cf., e.g., \cite[Theorem~18.7]{Leoni17}) and since $v=1$ on $\partial \Omega$, we get that $\mathrm{Tr}(u)_{|\partial \Omega}=0$.
\end{proof}

Actually, the above inclusions and properties  cha\-ra\-cte\-ri\-ze  $D^{1,2}(\Omega_c)$. 
In the following we give some useful properties of the homogeneous Sobolev space  $\stackrel{\bigcdot }{W}\!\!\mbox{}^{1,2}(\Omega_c)$ and a relation between it and $D^{1,2}(\Omega_c)$ which is crucial to our purposes.

\begin{rmk}\label{mainrmk}
Since $\Omega_c$ is a Lipschitz domain, the trace operator  is well-defined  on 
$\W(\Omega_c)$. Indeed, for any $R>0$ such that $\overline{\Omega}\subset B(0,R)$ we have $\W\big(B(0,R)\setminus\overline\Omega\big)=W^{1,2}\big(B(0,R)\setminus\overline\Omega\big)$
(see \cite[Corollary 1.1.11]{Mazya11})\footnote{Notice that a Lipschitz domain according to Definition~\ref{lipdom} is called $C^{0,1}$ domain in \cite{Mazya11}; $C^{0,1}$ bounded domains satisfy the cone property required in \cite[Corollary 1.1.11]{Mazya11} (see \cite[Remark 1, p. 15]{Mazya11}).} and then   $\mathrm{Tr}(u)\in L^2(\partial\Omega_c)$, for all $u\in\ \W(\Omega_c)$.  However, the trace operator is evidently not bounded in $\W(\Omega_c)$. Anyway, if $\|\nabla u_k-\nabla u\|_2\to 0$ and $\|u_k-u\|_{2^*}\to 0$, then $\|\mathrm{Tr}(u_k)-\mathrm{Tr}(u)\|_{L^2(\partial\Omega_c)}\to 0$. Moreover, if $u\in C^0(\overline \Omega_c)$, then $\mathrm{Tr}(u)=u_{|\partial \Omega}$. \end{rmk}
 \begin{prop}\label{trace}
The  space $D^{1,2}(\Omega_c)$ is  given by
$$
D^{1,2}(\Omega_c)=\{u\in \ \stackrel{\bigcdot}{W}\!\!\mbox{}^{1,2}(\Omega_c):  u\in L^{2^\ast}(\Omega_c), \ \mathrm{Tr}(u)=0\}.
$$
\end{prop}
\begin{proof}
The first inclusion has  already been  shown in Proposition~\ref{tr0}. For the other one,  let 
 $u\in \ \W(\Omega_c)\cap L^{2^\ast}(\Omega_c)$ be such that $\mathrm{Tr}(u)=0$.  As $u\in L^{2^*}(\Omega_c)$ and $|\nabla u|\in L^2(\Omega_c)$, for any $\epsilon>0$ there exists $R>0$ such that $\overline\Omega\subset B(0,R)$, $\|u\|_{L^{2^*}(\R^n\setminus B(0,R))}<\epsilon$ and 
 $\|\nabla u\|_{L^{2}(\R^n\setminus B(0,R))}<\epsilon$. Let us take $v$ as in the proof of Proposition~\ref{tr0}, with the further requirement that $|\nabla v|\leq 1/R$ on $A_R:=B(0,2R)\setminus B(0,R)$,  so that   $vu\in H^1_0(B(0,2R)\setminus \overline{\Omega})$. Let $(u_k)_k\subset C^{\infty}_c(B(0,2R)\setminus \overline{\Omega})$ converge to $vu$ in the $H^1$-norm. Then,   taking into account that $\|u\|_{L^{2^*}(A_R)}<\epsilon$ and $\|\nabla u_k\|_{L^2(A_R)}\to \|\nabla(vu)\|_{L^2(A_R)}$,  trivially  extending   $u_k$ on $\R^n\setminus B(0,2R)$, we get that   $u_k\to u$ in $D^{1,2}(\Omega_c)$. 
\end{proof}
 Let us now introduce the following subset of $D^{1,2}(\Omega_c)$:
$$\mathcal X_0:=D^{1,2}(\Omega_c)\cap  \{u\in C^{0,1}_{\mathrm{loc}}(\Omega_c): \|\nabla u\|_\infty\leq 1\}.$$ Let $j(u):\R^n\to \R$, $u\in \mathcal X_0$, be defined  as follows:  
\[j(u)(x)=\begin{cases}
u(x)&\text{if $x\in \Omega_c$}\\
0&\text{otherwise} . 
\end{cases}\] 
\begin{lem}\label{isometry}
 The map $j$ is an isometry from $\mathcal X_0$ (as a subset of the Hilbert space $D^{1,2}(\Omega_c)$) into the subset   $D^{1,2}(\R^n)\cap  \{u\in C^{0,1}(\R^n): \|\nabla u\|_{L^\infty(\R^n)}\leq 1\}$ of $D^{1,2}(\R^n)$.
\end{lem}  
\begin{proof}	
Let $u\in \mathcal X_0$; it is clear that $j(u)\in D^{1,2}(\R^n)$. Hence, $j(u)\in L^{2^*}(\R^n)$ and $|\nabla j(u)|\in L^2(\R^n)$; moreover, $|\nabla j(u)(x)|\leq 1$ for a.e. $x\in \R^n$. By Morrey's embedding theorem and we deduce that $j(u)\in L^{\infty}(\R^n)$, hence  $j(u)\in W^{1,\infty}(\R^n)$ and therefore it is Lipschitz. 
\end{proof}	

As a consequence of Lemma~\ref{isometry}, $\mathcal X_0$ can be identified with a subset  of $D^{1,2}(\R^n)\cap  \{u\in C^{0,1}(\R^n): \|\nabla u\|_{L^\infty(\R^n)}\leq 1\}$ and then by   
\cite[Lemma 2.1]{BoPo16}  we immediately get: 
\begin{lem}\label{prop}
As a  subset of the Hilbert space $D^{1,2}(\Omega_c)$, $\mathcal X_0$   satisfies the following properties:
	\begin{enumerate}[label=(\arabic{*}), ref=\arabic{*}]
		\item \label{p1}it is continuously embedded in $W^{1,p}(\Omega_c)$, for all $p\in [2^\ast, + \infty)$;
		\item\label{p2}
		it  is continuously embedded in $L^\infty(\Omega_c)$;
		\item\label{p3} all $u\in \mathcal X_0$ satisfy $\lim_{|x|\to \infty}u(x)= 0$;
		\item\label{p4}
		it is a convex and  weakly closed subset of $D^{1,2}(\Omega_c)$;
		\item\label{p5}
		any  bounded sequence $(u_k)_k\subset \mathcal X_0$ admits a subsequence weakly converging  to some $u\in \mathcal X_0$ and uniformly on compact subsets of $\Omega_c$.
	\end{enumerate}
\end{lem}
 
\bere\label{rem:ws}
Let us observe that if $u\in \mathcal{X}$  (cf. \eqref{storto}) is a weak solution according to Definition \ref{def:ws}, then the identity in \eqref{ws} also holds for any $v\in \mathcal{X}_0$ by means of a convolution argument and \cite[Lemma 2.10]{BoPo16}.
\ere
\begin{lem}\label{Lext}
Let $\Omega_c$ satisfy Assumption~\ref{omega}. Then, every $u\in C^{0,1}_{\mathrm{loc}}(\Omega_c)$ such that $|\nabla u|\in L^{\infty}(\Omega_c)$ can be extended to a  Lipschitz function on $\overline \Omega_c$. 
\end{lem}	
\begin{proof}
Let us take $R>0$ such that $\overline\Omega\subset B(0,R)$ and consider the bounded open set $V_R:=B(0,R)\setminus  \overline \Omega$. Let us  show that $u$ is bounded on $V_R$. Clearly $V_R$ is a Lipschitz domain and then it is the union of a finite number of Lipschitz domains  $V_j$ starshaped with respect to  balls $B(y_j, R_j)$ contained in $V_R$ (see \cite[Lemma 1, p. 15]{Mazya11}). Now assume by contradiction that $u_{|V_R}$ is not bounded. Then there exists a sequence  $(x_k)_k\subset V_R$ such that $|u(x_k)|\to +\infty$. Let $j_k$ be one of the indices $j$ such that $x_k\in V_{j_k}$. Then the segment $\overline{y_{j_k}x_{k}}$ is contained in $V_{j_k}$. Hence,
\[
|u(x_k)|\leq |u(y_{j_k})|
+\|\nabla u\|_\infty|x_{k}-y_{j_k}|\leq M+2R\|\nabla u\|_\infty,
\]
where $M:=\max_{j}|u(y_{j})|$, a contradiction. 

Thus, $u_{|V_R}\in W^{1,\infty}(V_R)$ and, since $V_R$ has the  extendibility property  (see \cite[Th. 13.17]{Leoni17}\footnote{We point out that the theorem can be applied since $	\partial\Omega=\partial\Omega_c$ is bounded, hence by \cite[p. 424]{Leoni17} being Lipschitz is equivalent to be uniformly Lipschitz.}), it admits an extension $\tilde u\in W^{1,\infty}(\R^n)$. Then 
\begin{equation}\label{exte}
\bar u(x):=\begin{cases} \tilde u(x)&\text{if $x\in B(0,R)$}\\
u(x)&\text{otherwise}\end{cases}
\end{equation}
is an extension of $u$ to $\R^n$ such that $|\nabla \bar u|\in L^{\infty}(\R^n)$; then it is Lipschitz on $\R^n$ (see, e.g. \cite[Ex. 11.50-(i)]{Leoni17})  and $\bar u_{|\overline \Omega_c}$ is a  Lipschitz extension of $u$.
\end{proof}

Now we are ready to show
that $\mathcal X$ is included in $W^{1,p}(\Omega_c)$, for all $p\in [2^*,+\infty]$.
\begin{lem}\label{sprimo}
Let  $u\in\mathcal X$; then $u\in W^{1,p}(\Omega_c)$, for all $p\in [2^*,+\infty]$ and moreover $\displaystyle\lim_{|x|\to\infty}u(x)=0$.
\end{lem}	
\begin{proof}
 As in the proof of Lemma~\ref{Lext},  $u$ admits a Lipschitz  extension $\bar u$ to $\R^n$. Since $\Omega$ is  bounded,  we get that 
%&=D^{1,2}(\R^n)\cap \{u\in C^{0,1}(\R^n):\|\nabla u\|_{L^\infty(\R^n)}\leq c(\partial\Omega)\}. 
\[\bar u \in \ \W(\R^n)\cap L^{2^*}(\R^n)\cap C^{0,1}(\R^n)=D^{1,2}(\R^n)\cap  C^{0,1}(\R^n). 
\]
 we deduce that $\bar u\in W^{1,p}(\R^n)$ for all $p\in [2^*,+\infty]$ and $\lim_{|x|\to\infty}\bar u(x)=0$;  plainly, analogous properties hold for~$u$.   
\end{proof}

Finally, let $\varphi\in L^{2}(\partial \Omega)$ be such that there exists $w\in\mathcal X$ with $\mathrm{Tr}(w)=\varphi$. By Lemma~\ref{Lext} we  have that $\varphi=w_{|\partial\Omega}$ (in particular $\varphi$ must be Lipschitz on $\partial\Omega$). 
Let us set
\[\X :=\{u\in\mathcal X:\mathrm{Tr}(u)=\varphi\}.\] 
Then by Proposition~\ref{trace} for every $w\in \X$ we get
\[\X=\{w\}+\mathcal X_0.\]

\begin{prop}\label{weakly}
 $\X$ is convex and weakly closed as   a topological subset  of  the semi-normed vector space $\W(\Omega_c)$.  Moreover,   if $(u_k)_k$ is a bounded sequence in $\X$, then up to a subsequence it weakly converges to some $u\in    \X $ and uniformly converges on compact subsets of $\Omega_c$.
\end{prop}
\begin{proof}
 Fix any $w\in \X$;  since  $\X$ is the translation by $w$ of a convex and weakly closed  subset of $D^{1,2}(\Omega_c)\hookrightarrow\ \W(\Omega_c)$ and the trace operator is linear, we get that $\X$ is convex and weakly closed in $\W(\Omega_c)$. 

If $(u_k)_k$ is a bounded sequence in $\mathcal X$, then $(u_k-w)_k$ is bounded in $D^{1,2}(\Omega_c)$, hence there exists a weakly  converging subsequence to a function $v\in D^{1,2}(\Omega_c)$. This implies that $u_k\rightharpoonup w+v\in \X$. Moreover,  by Lemma~\ref{prop}--\eqref{p2},  $(u_k-w)_k$ is bounded in $L^\infty(\Omega_c)$ and then, by Ascoli-Arzel\`a theorem,  for any given compact set $K\subset \Omega_c$  it admits a  subsequence uniformly converging to $v$ on $K$.  
\end{proof}

\section{Proofs of the main results}\label{s3}
Let $\varphi\in L^{2}(\partial \Omega)$ be such that there exists $w\in\mathcal X$ with $\mathrm{Tr}(w)=\varphi$ and consider the functional 
\beq\label{I}
{\mathcal{I}}(u)=
\int_{\Omega_c}\left(1-\sqrt{1-|\nabla u|^2}\right) {\rm d}x+ 
\int_{\Omega_c}G(x,u)\,{\rm d}x
\eeq
with $G(x,t):=n\int_0^tH(x,s)\;{\rm d} s$. Recalling that $\frac{1}{2}t\leq 1-\sqrt{1-t}\leq t$, for all  $t\in [0,1]$,     by  assumption \eqref{h} $\mathcal I$ is well-defined on  $\X$ (see also the proof of the first part of  Lemma \ref{deb}).
Notice also that every spacelike critical point $u$ of $\mathcal I$ weakly satisfies (in the sense of Definition \ref{def:ws}) the first equation in \eqref{star}. 
Thus, a  spacelike weak  solution of \eqref{star} can be found if  $\mathcal I$ has a spacelike minimizer on $\X$. In order to prove this last statement we need the following lemmas. 
Let us set
$$
\mathcal I_0(u):=\int_{\Omega_c}\left(1-\sqrt{1-|\nabla u|^2}\right){\rm d}x\quad\text{and}\quad \mathcal{G}(u):=\int_{\Omega_c}G(x,u)\, {\rm d}x. 
$$

\begin{lem}\label{deb}
Under the assumption on $\Omega_c$ and $H$ in Theorem~\ref{maint}, assume also that $\varphi=\mathrm{Tr}(w)$ for some $w\in\mathcal X$. Then  
 $\mathcal{G}$ is well-defined on $\X$ and  sequentially  weakly continuous.
\end{lem}
\begin{proof}
Let us  denote  by $s'$ the conjugate exponent of $s$.   As  $u\in \X\subset \mathcal X$,  by Lemma~\ref{sprimo} we get that $u\in W^{1,p}(\Omega_c)$ for all $p\in [2^*,+\infty]$.
Then for all  $s\in [1,\frac{2n}{n+2}]$,  $s'\geq 2^\ast$ and
$$
|\mathcal{G}(u)|\leq \|h\|_s\|u\|_{s'}<+\infty.
$$
Let us now show that $\mathcal G$ is sequentially weakly continuos on $\X$. Assume that  $(u_k)_k\subset \X$ weakly converges to $u\in \ \W(\Omega_c)$. By Proposition~\ref{weakly}, $u\in\X$;
moreover $(u_k-w)_k\subset \mathcal X_0$ and then it is bounded in $D^{1,2}(\Omega_c)$.
Since $h\in L^s(\Omega_c)$, for a given $\epsilon>0$ there exists $R>0$ such that $\overline\Omega\subset B(0,R)$ and
$$
\|h\|_{L^{s}(\R^n\setminus \overline{B(0,R)})}<\epsilon;
$$
therefore by  Lemma \ref{prop}--\eqref{p1}, there exists $C_1>0$ such that 
\begin{align}
\int_{\R^n\setminus \overline{B(0,R)}}|G(x,u_k)|\;{\rm d}x
&\leq \|h\|_{L^{s}(\R^n\setminus \overline{B(0,R)})}\|u_k\|_{s'}
\nonumber
\\
&\leq
\|h\|_{L^{s}(\R^n\setminus \overline{B(0,R)})}\left(\|u_k-w\|_{s'} + \|w\|_{s'}\right)
%\\&
\leq C_1 \epsilon \label{pippo}
\end{align}
and
\begin{equation}\label{pippa}
\int_{\R^n\setminus \overline{B(0,R)}}|G(x,u)|\;{\rm d}x
\leq C_1 \epsilon.
\end{equation}
Notice that, again by assumption \eqref{h}, taking the bounded subset $B(0,R)\setminus \overline\Omega$, it results in particular $h\in L^1(B(0,R)\setminus\overline\Omega)$ and then 
by  Lemma \ref{prop}--\eqref{p2}, for a $C_2>0$ we get
$$
|G(x,u_k)|\leq h(x)\|u_k\|_{\infty} \leq h(x)(\|u_k-w\|_{\infty}+\|w\|_{\infty})\leq C_2 h(x),
$$ 
for all  $k\in\N$,  a.e.  in $B(0,R)\setminus\overline\Omega$. 
By Proposition~\ref{weakly} the bounded sequence $(u_k)_k$ uniformly converges on compact subsets in $\Omega_c$, up to a subsequence. Hence, $(G(x,u_k))_k$ converges to $G(x,u)$ a.e. in $B(0,R)\setminus\overline\Omega$ and  by the Lebesgue's dominated convergence theorem
it follows that
$$
\int_{B(0,R)\setminus\overline\Omega}G(x,u_k)\;{\rm d}x \longrightarrow
\int_{B(0,R)\setminus\overline\Omega}G(x,u)\;{\rm d}x, \quad\hbox{as } k\longrightarrow +\infty.
$$
As the last  convergence actually holds for the whole sequence, being $\epsilon$ arbitrary, by \eqref{pippo} and \eqref{pippa}, we are done. 
\end{proof}
Let us now show that $\mathcal I$ has a global minimum point on $\X$.

\begin{prop}\label{pr:min}
Under the assumption on $\Omega$ and $H$ in Theorem~\ref{maint}, the functional $\mathcal{I}$ possesses at least a minimizer in $\mathcal{X}_\varphi$.
\end{prop}

\begin{proof}
By assumption \eqref{h},  Lemma~\ref{sprimo} and \eqref{p1}-\eqref{p2} of Lemma~\ref{prop}, denoting by $s'$ the conjugate exponent of $s$, for all $u\in\X$ we have:
\baln
\mathcal{I}(u)&\geq \frac{1}{2}\|\nabla u\|^2_2-\|h\|_s\|u\|_{s'}\\
&\geq\frac{1}{2}\|\nabla u\|^2_2-\|h\|_s\|u-w\|_{s'}-\|h\|_s\|w\|_{s'}\\ 
& \geq \frac{1}{2}\|\nabla u\|^2_2-C_1\|h\|_s\|\nabla u-\nabla w\|_{2}-\|h\|_s\|w\|_{s'}\\ 
&\geq \frac{1}{2}\|\nabla u\|^2_2-C_1\|h\|_s\|\nabla u\|_2-C_1\|h\|_s\|\nabla w\|_{2}-\|h\|_s\|w\|_{s'}
\ealn
hence $\mathcal{I}$ is coercive.  By Proposition~\ref{weakly}, in order to get the existence of a minimizer, we just need to prove that $\mathcal{I}$ is sequentially weakly lower semi-continuous. Actually this holds because the first term $\mathcal{I}_0$ of $\mathcal I$ is convex and strongly continuous (cf.  \cite[Lemma 2.2]{BoPo16} for details), thus it is  sequentially  weakly lower semi-continuous in the semi--normed space $\W(\Omega_c) $
 and  $\mathcal{G}$ is  sequentially  weakly continuous by  Lemma~\ref{deb}.
\end{proof}
The lemma below is based on \cite{BM} and shows that a minimizer $u$ of $\mathcal I$ is a  minimizer also for the analogous  functional  corresponding to  $\mathcal I$ with  $H^*(x)=nH(x,u(x))$ replacing $nH$. 
\begin{lem}\label{min}
	Under the assumptions of Lemma~\ref{deb}, assume also that $u$ is a minimizer of $\mathcal I$ on $\X$ and set $H^*(x):=nH(x,u(x))$. Then $u$ is  also  a minimizer   of the functional $\mathcal I^*\colon \X\to \R$,  where  
	\[\mathcal I^*(v):=\int_{\Omega_c}\left(1-\sqrt{1-|\nabla v|^2}\right){\rm d}x+\int_{\Omega_c}H^*(x)v(x)\,\de x.\]  
\end{lem} 
\begin{proof}
	Let $v\in \X$; then   $u_\lambda:=u+\lambda(v-u)\in \X$, for all $\lambda\in[0,1]$. Hence for all $\lambda\in (0,1]$
	$$\mathcal I_0(u_\lambda)+\mathcal G(u_\lambda)\geq \mathcal I_0(u)+\mathcal G(u)
	$$
	and    being  $\mathcal I_0$ convex  we obtain
	\beq\label{withlambda}
	\mathcal I_0(v)-\mathcal I_0(u)-\frac 1 \lambda \big(\mathcal G(u_\lambda)-\mathcal G(u)\big)\geq 0.
	\eeq
	Let $\sigma(x)\in [0,1]$ be such that 
	\[
	\frac 1 \lambda \big(\mathcal G(u_\lambda)-\mathcal G(u)\big)= n \int_{\Omega_c}H\big(x,u+\sigma(x)\lambda(v-u)\big)(v-u)\ \de x;
	\]
	hence, by assumption \eqref{h} and Lebesgue's dominated convergence theorem (recall that $v-u\in \mathcal X_0\hookrightarrow L^{\infty}(\Omega_c)$), we get 
	\[ n \int_{\Omega_c}H\big(x,u+\sigma(x)\lambda(v-u)\big)(v-u)\ \de x\longrightarrow \int_{\Omega_c}H^*(x)(v-u)\ \de x,\]
	 as $\lambda\to 0$, which by \eqref{withlambda} implies 
	\[\mathcal I^*(v)-\mathcal I^*(u)\geq 0.\]
\end{proof}

\bere
Let $\bar u\in \mathcal{X}_\varphi$ be a minimizer of $\mathcal{I}$ found in Proposition \ref{pr:min}. By Lemma \ref{min} all the conclusions of \cite[Proposition 2.7]{BoPo16}, suitable modified, hold and, in particular, we have that 
\[
{\rm meas}\{x\in \Omega_c: |\nabla \bar u|=1\}=0
\]
and 
\[
\frac{|\nabla \bar u|^2}{\sqrt{1-|\nabla \bar u|^2}}\in L^1(\Omega_c).
\]
However, this is still not enough to conclude that $\bar u$ is a weak solution of problem \eqref{star} in the sense of Definition \ref{def:ws}.

Finally, we point out that,  by \cite[Proposition~1.1]{BarSim83}, the minimizer is unique if
$H$ is non-decreasing in  $t$.

\ere

We can now conclude the proof of our main result.
\begin{proof}[Proof of Theorem \ref{maint}.]
As a first step we want to prove that if $\varphi$ is the trace of a function in $\mathcal X_\varphi$, then each minimizer $\bar u$ of $\mathcal{I}$ in $\mathcal{X}_\varphi$ is not only weakly spacelike, but  spacelike as well, and so it is a weak solution of \eqref{star}. 

By Lemma~\ref{min} $\bar u$ is a minimizer for the functional $\mathcal I^*$ and, since  $H$  is locally bounded by   \cite[Theorem 3.2]{BarSim83}, we infer that any segment  of a lightlike geodesic possibly contained in the graph of $\bar u$ can be extended until it reaches a point on the graph of $\varphi$. This fact and the local estimates in the proof of  \cite[Theorem 4.1]{BarSim83} imply that  the subset  $K\subset \Omega_c$ where a minimizer could be non-regular  is precisely given by the points which are the projections on $\Omega_c$ of light rays  and lightlike segments (i.e., resp.,  lines or half-lines and  segments whose tangent vectors are  lightlike) contained in the graph $G_{\bar u}$, such that, respectively, no point or at least one  point or both endpoints  belong to $\partial\Omega_c$. Thus, if $\Omega$ is convex, $K$ might  contain only lines or half-lines (without, potentially, their points in the boundary) which are projections of light rays in $G_{\bar u}$, but this is incompatible with the fact that $\lim_{|x|\to\infty}\bar u(x)=0$ (recall Lemma~\ref{sprimo}). If $\Omega$ is not convex and there exists a segment $\overline{xy}$ contained in $K$ whose endpoints $x,y$ belong to $\partial \Omega$,  we would have   
$|\varphi(x)-\varphi(y)|=|\bar u(x)-\bar u(y)| = |x-y|$, in contradiction with the spacelike displacing assumption on $\varphi$.
Therefore,  $K=\emptyset$ and $\bar u$ is spacelike.

Let us now prove the reverse implication of the theorem.
We observe that, according to Definition~\ref{def:ws}, any spacelike weak  solution $\bar u$ of \eqref{star} 
is locally $1$--Lipschitz, then by Lemma~\ref{Lext} $\varphi=\bar u_{|\partial\Omega}$ and if $\Omega$ is not convex and there exists a segment $\overline{xy}$ in $\Omega_c$ connecting two points $x, y\in\partial \Omega$, then $|\varphi(x)-\varphi(y)|=|\bar u(x)-\bar u(y)|<|x-y|$, i.e., $\varphi$ is spacelike displacing in $\Omega_c$.  

\end{proof}
If $\varphi$ is $(1-\epsilon)$--Lipschitz, we can show that it is the trace of a function in $\mathcal X_\varphi$.
\begin{proof}[Proof of Corollary~\ref{mainc}]
We can extend $\varphi$ to a bounded $(1-\epsilon)$--Lipschitz function $\psi$ on $\R^n$ such that $\min\psi=\min\varphi$ and $\max\psi=\max\varphi$ (see, e.g., \cite[p. 243]{Leoni17}). Let us take $R>0$ such that $\overline{\Omega}\subset B(0,R)$, $\epsilon R >\|\varphi\|_\infty$ and a smooth function 
$v:\R^n\to [0,1]$ with compact support in $B(0,2R)$,  equal to $1$ on $B(0,R)$ and such that $|\nabla v|\leq 1/R$. As $|\nabla(v\psi)|\leq \frac{\|\varphi\|_\infty}{R}+1-\epsilon<1$, the restriction of $v\psi$ to $\Omega_c$ belongs to $\X$. Therefore $\varphi$ is the trace of a function in $\mathcal X_\varphi$ and we can conclude by Theorem \ref{maint}.
\end{proof}

\end{document}